\documentclass[12pt]{amsart}

 \usepackage{amsfonts}
\usepackage{amssymb}
 \usepackage{amsthm}

\newtheorem{theorem}{Theorem}[section]
\newtheorem{lemma}[theorem]{Lemma}

\newtheorem{problem}[theorem]{Problem}
\newtheorem{conjecture}[theorem]{Conjecture}
\theoremstyle{definition}

\theoremstyle{remark}

\numberwithin{equation}{section}

\begin{document}

\title[Coverings of commutators in profinite groups]{Coverings of commutators in profinite groups}
\author{Cristina Acciarri}

\address{Cristina Acciarri:  Department of Mathematics, University of Brasilia,
Brasilia-DF, 70910-900 Brazil}
\email{acciarricristina@yahoo.it}

\author{Pavel Shumyatsky} 

\address{Pavel Shumyatsky: Department of Mathematics, University of Brasilia,
Brasilia-DF, 70910-900 Brazil}

\email{pavel@unb.br}


\thanks{This work was supported by  FAPDF and the Conselho Nacional de Desenvolvimento
Cient\'ifico e Tecnol\'ogico (CNPq), Brazil}

\begin{abstract}  Let $w$ be a group-word. Suppose that the set of all $w$-values in a profinite group $G$ is contained in a union of countably many subgroups. It is natural to ask in what way the structure of the verbal subgroup $w(G)$ depends on the properties of the covering subgroups. The present article is a survey of recent results related to that question. In particular we survey results on finite and countable coverings of word-values (mostly commutators) by procyclic, abelian, nilpotent, and soluble subgroups, as well as subgroups with finiteness conditions. The last section of the paper is devoted to relation of the described results with Hall's problem on conciseness of group-words.

\end{abstract}

\maketitle

\section{Introduction}
A  covering of a group $G$ is a family $\{S_i\}_{i\in I}$ of subsets of $G$ such that $G=\bigcup_{i\in I}\,S_i$. When a group $G$ is covered by finitely many subgroups, it is natural to expect that some structural information about $G$ can be deduced from the properties of the covering subgroups. In the literature there are many results in this direction. An important combinatorial tool for dealing with problems of that kind is B.H. Neumann's Lemma \cite{neumann}: if $\{S_i\}$ is a finite covering of $G$ by cosets of subgroups, then $G$ is actually covered by the cosets $S_i$ corresponding to subgroups of finite index in $G$. In other words, we can get rid of the cosets of subgroups of infinite index without losing the covering property.

In recent years some ``verbal" variations of the above questions became a subject of research activity. Given a group-word $w=w(x_1,\dots,x_n)$, we think of it primarily as a function of $n$ variables defined on any given group $G$. We denote by $w(G)$ the verbal subgroup of $G$ generated by the values of $w$. When the set of all $w$-values in a group $G$ is contained in a union of finitely many subgroups, it is natural to ask whether the structure of the verbal subgroup $w(G)$ depends on the properties of the covering subgroups.

The questions become even more intriguing when the groups are profinite. In that case the interplay of algebraic and topological considerations adds a new quality to the obtained results.  A profinite group is a topological group that is isomorphic to an inverse limit of finite groups. The textbooks \cite{riza} and \cite{wils} provide a good introduction to the theory of profinite groups. In the context of profinite groups all the usual concepts of group theory are interpreted topologically. In particular, by a subgroup of a profinite group we mean a closed subgroup. A subgroup is said to be generated by a set $S$ if it is topologically generated by $S$. Thus, the verbal subgroup $w(G)$ in a profinite group $G$ is a minimal closed subgroup containing the set of $w$-values. One important tool for dealing with  the ``covering" problems in profinite groups is the classical Baire's category theorem (cf \cite[p.\ 200]{kell}): If a locally compact Hausdorff space is a union of countably many closed subsets, then at least one of the subsets has non-empty interior. It follows that if a profinite group is covered by countably many subgroups, then at least one of the subgroups is open. Thus,  in the case of profinite groups we have tools to deal with problems on countable coverings. Remark however that the Baire category theorem does not provide a full countable analogue of the Neumann lemma. A profinite group that is covered by countably many subgroups $\{H_i\}$ is not necessarily covered by open ones from the family $\{H_i\}$. For example, let $C_p$ denotes the cyclic group of prime order $p$ and $\mathbb{Z}_p$ the additive group of $p$-adic integers. Then $C_p\times\mathbb{Z}_p$ is a profinite group covered by countably many procyclic subgroups but of course $C_p\times\mathbb{Z}_p$ cannot be covered by open procyclic subgroups.

The present article is a survey of recent results on finite and countable coverings of word-values (mostly commutators) in profinite groups. Of course, finite coverings of abstract groups are also mentioned here since in many cases they serve as a starting point for the study of coverings in profinite groups. On the other hand, it must be said that several important lines of research on the relationship between $G_w$ and $w(G)$ are not featured in this paper. In particular, the reader will not find here any discussion of results on the verbal width (cf \cite{segalbook}) and problems of Waring type (cf \cite{waring}).

In Section 2 we describe known results and open problems on profinite groups in which commutators are covered by countably many procyclic subgroups. Section 3 is devoted to the case where the covering subgroups are abelian, nilpotent, or soluble.  Section 4 provides an account on coverings of word-values by subgroups with classical finiteness conditions -- in particular by periodic subgroups, or subgroups of finite rank. In the last section we discuss a recently discovered phenomenon closely related to the famous Hall problem whether the verbal subgroup $w(G)$ must be finite whenever the word $w$ has only finitely many values in a group $G$ (the problem was solved in the negative by Ivanov \cite{Ivan}). We describe results showing that in some cases $w(G)$ must be finite whenever the word $w$ has only countably many values in a profinite group $G$. The paper contains many open problems and conjectures related to the topic of covering of word-values by subgroups.

We thank the unknown referee for a number of useful suggestions on the earlier version of the article.

\subsection*{Notation} 
Given a profinite group $G$, we denote by $\pi(G)$ the set of prime divisors of the orders of finite continuous homomorphic images of $G$.  We say that $G$ is a $\pi$-group if $\pi(G)\subseteq\pi$ and $G$ is a $\pi'$-group if $\pi(G)\cap\pi=\emptyset$. If $m$ is an integer, then we denote by $\pi(m)$ the set of prime divisors of $m$. If $\pi$ is a set of primes, we denote by $O_\pi(G)$ the maximal normal $\pi$-subgroup of $G$ and by $O_{\pi'}(G)$ the maximal normal $\pi'$-subgroup.
 
As usual, the expression ``$(a,b,c,\ldots)$-bounded'' means ``bound\-ed from above by some function which depends only on the parameters $a,b,c,\ldots$''.
 
A word $w=w(x_{1},\ldots,x_{n})$ is said to be non-commutator if the sum of exponents of at least one variable $x_{i}$ involved in $w$ is non-zero. Given two arbitrary elements $x$ and $y$ of a group $G$, their commutator is defined as the element $[x,y]=x^{-1}y^{-1}xy$. The verbal subgroup of $G$ corresponding to the word $[x_{1},x_{2}]$ is the commutator subgroup $G'$. The lower central words $\gamma_k$, are defined recursively  by
\[
\gamma_1=x_1,
\qquad
\gamma_k=[\gamma_{k-1},x_k]=[x_1,\ldots,x_k],
\quad
\text{for $k\ge 2$.}
\]
The corresponding verbal subgroups $\gamma_k(G)$ are the terms of the lower central series of $G$.  Another family of words that generalize the simple commutator is that of the derived words $\delta_k$, on $2^k$ variables, which are defined recursively by 
\[
\delta_0=x_1,
\quad
\delta_k=[\delta_{k-1}(x_1,\ldots,x_{2^{k-1}}),\delta_{k-1}(x_{2^{k-1}+1},\ldots,x_{2^k})],\quad
\text{for $k\ge 1$.}
\] 
The verbal subgroup that corresponds to the word $\delta_k$ is the familiar $k$th derived subgroup of $G$ usually denoted by $G^{(k)}$.

Many results described in this article deal with so called {\it multilinear commutators}, also known under the name of {\it outer commutator words}.  These are the words that have a form of a multilinear Lie monomial, i.e., the words that are constructed by nesting commutators. The number of indeterminates used in the expression of a multilinear commutator word $w$ is called the weight of $w$.  For instance the word  $$[[x_1,x_2],[y_1,y_2,y_3],z]$$ is a multilinear commutator of weight six. Other examples of multilinear commutators are provided by the aforementioned lower central words $\gamma_k$ of weight $k$ and  derived words $\delta_k$ of weight $2^{k}$. The Engel word $[x_1,x_2,x_2,x_2]$ is an example of a commutator word which is not a multilinear commutator. 

\section{Coverings by procyclic subgroups} 
\subsection*{On groups covered by few cyclic subgroups}

It was pointed out by Baer (cf \cite[p.\ 105]{Robinson}) that an abstract group covered by finitely many cyclic subgroups is either cyclic or finite. Somewhat surprizingly, it seems that until recently there was no published work on finite groups covered by few cyclic subgroups. That gap was filled in \cite{AS3}. In particular, the article \cite{AS3} contains the following results.

1. Let $G$ be a finite noncyclic $p$-group covered by $m$ cyclic subgroups. Then the order of $G$ is $m$-bounded.

2. Let $G$ be a finite group covered by $m$ cyclic subgroups. Then $G$ has a normal subgroup $M$ of $m$-bounded order with the property that $G/M$ is cyclic.

Combining the above results it is not difficult to deduce the following theorem.
\begin{theorem}\label{1} If a finite group $G$ is covered by $m$ cyclic subgroups, then $G$ is a product of a normal subgroup $M$ of $m$-bounded order and a cyclic subgroup $H$ such that $(|M|,|H|)=1$. Conversely, assume that a finite group $G$ is a product of a normal subgroup $M$ and a cyclic subgroup $H$ such that $(|M|,|H|)=1$. Then $G$ is covered by $|M|$-boundedly many cyclic subgroups.
\end{theorem}

\begin{proof}[Sketch of the proof.] Let $m\geq 1$ and assume that $G$ is a finite group covered by $m$ cyclic subgroups. Let us show that $G$ is a product of a normal subgroup $M$ of $m$-bounded order and a cyclic subgroup $H$ such that $(|M|,|H|)=1$.

Let $K$ be the maximum of orders of noncyclic Sylow $p$-subgroups of $G$. We know that $K$ is bounded in terms of $m$ only. Further, $G$ has a normal subgroup $N$ of $m$-bounded order with the property that $G/N$ is cyclic. Without loss of generality we assume that $N\neq1$. Denote by $\pi$ the set of all primes for which the order of a Sylow $p$-subgroup of $G$ is at least $K|N|[G:C_G(N)]$. It follows that for any $p\in\pi$ the Sylow $p$-subgroup $P$ of $G$ is cyclic and has the property that $P\cap N\leq Z(N)$. Let $N_1$ be the subgroup  generated by all Sylow $p$-subgroups of $N$, where $p\in\pi$. It follows that $N_1\leq Z(N)$ and $N_1=O_\pi(N)$. Therefore $N=N_1\times N_2$, where $N_2=O_{\pi'}(N)$.

Suppose now that $N$ is actually a $p$-group for some $p\in\pi$. Let $P$ be the Sylow $p$-subgroup of $G$. It is clear that $P$ is normal. We will now use the fact that whenever $T$ is a cyclic $p$-group and $\alpha$ a $p'$-automorphism of $T$ we either have $T=[T,\alpha]$ or $\alpha=1$. It follows that in our situation either $P\leq G'$ or $P\leq Z(G)$. The former inclusion is impossible since $|G'|\leq|N|$ while $|P|>|N|$. Thus $P\leq Z(G)$. It follows that $G$ is nilpotent. Since $G'\leq P$ and $P$ is cyclic, we conclude that $G$ is actually abelian.

Thus, we showed that if $N$ is a $p$-group for some $p\in\pi$, then $G$ is abelian. Applying this argument to $G/O_{p'}(N)$ for each $p\in\pi$ we conclude that $G'\leq N_2$. Therefore all $\pi'$-elements of $G$ are contained in $O_{\pi'}(G)$. We set $M=O_{\pi'}(G)$. It follows from the definition of $\pi$ and the fact that $G'\leq M$ that the order of $M$ is $m$-bounded and $G/M$ is a cyclic $\pi$-group. By the Schur-Zassenhaus theorem $G$ possesses a cyclic $\pi$-subgroup $H$ such that $G=MH$.

Now assume that $G$ is a product of a normal subgroup $M$ and a cyclic subgroup $H$ such that $(|M|,|H|)=1$. Let us show that $G$ can be covered by $m$ cyclic subgroups for some $|M|$-bounded number $m$. Set $\pi=\pi(M)$ and denote by $X_1$ the set of all $\pi$-elements in $G$, by $X_2$ that of $\pi'$-elements, and by $X_3$ the set of elements which are neither $\pi$- nor $\pi'$-elements.

It is clear that $X_1=M$ and so this set is covered by at most $|M|$ cyclic subgroups.

By the Schur-Zassenhaus theorem $X_2$ is a union of (cyclic) Hall $\pi'$-subgroups. Since the Hall $\pi'$-subgroups are conjugate, there are at most $|M|$ of them and so $X_2$ is covered by at most $|M|$ cyclic subgroups.

Choose a generator $a$ of $H$. Let $n$ be the minimal number such that $a^n\in Z(G)$. Of course, $n$ is bounded by $|M|$. Set $M_i=C_M(a^i)$ for each $i=1,2,\dots$. Of course, $M_n=M$. It is clear that each subgroup of the form $\langle a^i,M_i\rangle$ is covered by at most $|M_i|$ cyclic subgroups. Each of these subgroups has at most $[M:M_i]$ conjugates in $G$. It remains to observe that each element in $X_3$ belongs to a conjugate of a subgroup of the form $\langle a^i,M_i\rangle$. So the set $X_3$ is covered by $|M|$-boundedly many cyclic subgroups.
\end{proof}
\bigskip

\subsection*{On groups in which commutators are covered by few cyclic subgroups}

In \cite{FAS} Fern\'andez-Alcober and Shumyatsky  showed that if $G$ is an abstract group in which the set of all  commutators is covered by finitely many cyclic subgroups, then the commutator subgroup $G'$ either is finite or cyclic. Further information on the structure of such groups was obtained in \cite{FAMS} where the following result was proved.
\begin{theorem} \label{2} Let $G$ be a group that possesses $m$ cyclic subgroups whose union contains all commutators of $G$. Then $G$ has a characteristic subgroup $M$ contained in $G'$ such that the order of $M$ is $m$-bounded and $G'/M$ is cyclic.
\end{theorem}
Of course, the above theorem is meaningful only in the case where $G'$ is finite, since otherwise $G'$ is cyclic. Comparing Theorem \ref{2} with Theorem \ref{1} it is natural to wonder whether further details on the structure of $G'$ in Theorem \ref{2} can be obtained.

\begin{problem} \label{3} Let $G$ be a group that possesses $m$ cyclic subgroups whose union contains all commutators of $G$ and suppose that $G'$ is finite. Is $G'$ a product of a normal subgroup $M$ of $m$-bounded order and a cyclic subgroup $H$ such that $(|M|,|H|)=1$?
\end{problem}

\bigskip
\subsection*{Further results on coverings by few cyclic subgroups}

The above results suggest questions on the structure of a group in which the set of all $w$-values is covered by finitely many cyclic subgroups. In \cite{CN} Cutolo and Nicotera showed that if $w$ is a non-commutator word and $G$ is an abstract group in which the set of all $w$-values is covered by finitely many cyclic subgroups, then $w(G)$ is finite-by-cyclic. Further, they showed that if the set of all $\gamma_{k}$-commutators in $G$ is covered by finitely many cyclic subgroups, then $\gamma_{k}(G)$ is finite-by-cyclic. They also showed that $\gamma_{k}(G)$ can be both noncyclic and infinite. It seems reasonable to expect that the structure of $\gamma_{k}(G)$ depends on the number of cyclic subgroups covering the set of $\gamma_{k}$-commutators.

\begin{problem} \label{4} Let $k$ be a positive integer and $G$ a group that possesses $m$ cyclic subgroups whose union contains all $\gamma_k$-commutators. Is $\gamma_{k}(G)$ a product of a normal subgroup $M$ of $m$-bounded order and a cyclic subgroup $H$?
\end{problem}

It is also natural to ask if the result in \cite{CN} can be extended to arbitrary multilinear commutators.

\begin{problem} \label{5} Let $w$ be a multilinear commutator word and $G$ a group that possesses finitely many cyclic subgroups whose union contains all $w$-values. Is $w(G)$ finite-by-cyclic?
\end{problem}

It is clear that the above question has negative answer if we do not require $w$ to be a multilinear commutator word. The following example was given in \cite{CN}. Ivanov constructed in \cite{Ivan} a group $K$ that admits a word $v$ such that $v$ has only one nontrivial value in $K$ while the subgroup $v(K)$ is infinite cyclic. Let $G=K\times K$. It is clear that $v$ has only 3 nontrivial values in $G$ and therefore the set of $v$-values in $G$ is covered by 3 cyclic subgroups. On the other hand, the verbal subgroup $v(G)$ is not finite-by-cyclic. 

\bigskip
\subsection*{On coprime commutators}

The coprime commutators $\gamma_k^*$ and $\delta_k^*$ in a finite group were introduced in \cite{Sh1}. The definition goes as follows. 

Every element of a finite group $G$ is  both a $\gamma_1^*$-commutator and a $\delta_0^*$-commutator. Now let $k\geq 2$ and let $X$ be the set of all elements of $G$ that are powers of $\gamma_{k-1}^*$-commutators. An element $g$ is a $\gamma_k^*$-commutator if there exist $a\in X$ and $b\in G$ such that $g=[a,b]$ and $(|a|,|b|)=1$. For $k\geq 1$ let $Y$ be the set of all elements of $G$ that are powers of $\delta_{k-1}^*$-commutators. The element $g$ is a $\delta_k^*$-commutator if there exist $a,b\in Y$ such that $g=[a,b]$ and $(|a|,|b|)=1$. One can easily see that if $N$ is a normal subgroup of $G$ and $x$ an element whose image in $G/N$ is a $\gamma_k^*$-commutator (respectively a $\delta_k^*$-commutator), then there exists a $\gamma_k^*$-commutator $y$ in $G$ (respectively a $\delta_k^*$-commutator) such that $x\in yN$. The subgroups of $G$ generated by all $\gamma_k^*$-commutators and all  $\delta_k^*$-commutators are denoted by $\gamma_k^*(G)$ and $\delta_k^*(G)$, respectively. It was shown in \cite{Sh1} that for every $k\geq2$ the subgroup $\gamma_k^*(G)$ is precisely the last term of the lower central series of $G$ (which is usually denoted by $\gamma_\infty(G)$) while for every $k\geq1$ the subgroup $\delta_k^*(G)$ is precisely the last term of the lower central series of $\delta_{k-1}^*(G)$, that is, $\delta_k^*(G)=\gamma_\infty(\delta_{k-1}^*(G))$. In other terminology, $\gamma_k^*(G)$ is precisely the nilpotent residual and $\delta_k^*(G)$ is precisely the $k$th term of the lower Fitting series of $G$.

In the context of the above discussion the following results obtained in \cite{AS3} are very natural.
\begin{theorem}
\label{cop1} Let  $k$ be a positive integer and $G$  a finite group that possesses $m$ cyclic subgroups whose union contains all $\gamma_k^*$-commutators of $G$. Then $\gamma_k^*(G)$ contains a subgroup $M$, of $m$-bounded order, which is normal in $G$ and has the property that $\gamma_{k}^{*}(G)/M$ is cyclic.
\end{theorem}

\begin{theorem}\label{cop2} Let  $k\geq2$ and $G$ be a finite group that has $m$ cyclic subgroups whose union contains all $\delta_k^*$-commutators of $G$.  Then the order of $\delta_k^*(G)$ is $m$-bounded.
\end{theorem}

\subsection*{On groups covered by countably many procyclic subgroups}

Using the inverse limit argument the results on finite groups covered by few cyclic subgroups can be easily extended to profinite groups covered by finitely many procyclic subgroups. We therefore have

\begin{theorem}\label{11} A profinite group $G$ can be covered by finitely many procyclic subgroups if and only if $G$ is a product of a normal finite subgroup $M$ and a procyclic subgroup $H$ such that $\pi(M)\cap\pi(H)=\emptyset$.
\end{theorem}

This shows that Baer's result that an abstract group covered by finitely many cyclic subgroups is either cyclic or finite cannot be extended to profinite groups. On the other hand, Theorem \ref{11} implies that we have an analogue of Baer's result for pro-$p$ groups: a pro-$p$ group covered by finitely many procyclic subgroups must be either procyclic or finite.

The following theorem obtained in \cite{AS4} describes profinite groups covered by countably many procyclic subgroups.

\begin{theorem}\label{10} A profinite group $G$ can be covered by countably many procyclic subgroups if and only if $G$ is finite-by-procyclic.
\end{theorem}

Thus, we have a pretty clear description of profinite groups covered by countably many procyclic subgroups and we know when such a group is actually covered by only finitely many procyclic subgroups.

\bigskip
\subsection*{On groups in which commutators are covered by countably many procyclic subgroups}

Profinite groups in which the set of commutators is covered by countably many procyclic subgroups were studied in \cite{AS4}. The groups were characterized as follows.

\begin{theorem}\label{pro3}
Let $G$ be a profinite group. The set of all commutators of $G$ is contained in a  union of countably many procyclic subgroups if and only if the commutator subgroup $G'$ is finite-by-procyclic.
\end{theorem}

Specific questions on the structure of a profinite group in which the set of commutators is covered by finitely many procyclic subgroups were addressed in \cite{FAMS}, where the following results were proved.

\begin{theorem}\label{12} If $G$ is a profinite group in which all commutators are covered by $m$ procyclic subgroups, then $G$ possesses a finite characteristic subgroup $M$ contained in $G'$ such that the order of $M$ is $m$-bounded and $G'/M$ is procyclic.
\end{theorem}
\begin{theorem}\label{13} If $G$ is a pro-$p$ group in which all commutators are covered by $m$ procyclic subgroups, then $G'$ is either finite of $m$-bounded order or procyclic.
\end{theorem}

Several observations come to mind in the context of the above results. We remark that in the case of pro-$p$ groups there is a clear way to distinguish between the groups in which the commutators are covered by finitely many procyclic subgroups and those in which the  covering requires infinitely many ones. This is not so in the profinite case. Based on our previous experience, we make the following conjecture.
\begin{conjecture}\label{15} Let $G$ be a profinite group whose commutator subgroup is finite-by-procyclic. The commutators in $G$ are covered by finitely many procyclic subgroups if and only if $G'$ is a product of a normal finite subgroup $M$ and a procyclic subgroup $H$ such that $\pi(M)\cap\pi(H)=\emptyset$.
\end{conjecture}

A noteworthy fact that can be easily deduced from Theorem \ref{pro3} is that if the set of commutators in a pro-$p$ group $G$ is contained in a union of countably many procyclic subgroups of $G'$, then at least one of the subgroups is open in $G'$.

Finally, we remark that if the set of commutators in a profinite group $G$ is contained in a union of countably many procyclic subgroups, then the entire commutator subgroup $G'$ admits a covering by countably many procyclic subgroups. This follows from Theorems \ref{pro3} and \ref{10}. But of course we cannot claim that the family of procyclic subgroups that covers the set of all commutators is necessarily the same as the one that covers $G'$.

\section{Coverings by abelian, nilpotent, and soluble subgroups}

 \subsection*{On groups covered by few nilpotent subgroups}

Suppose that a group $G$ is covered by finitely many abelian subgroups. In view of Neumann's lemma we can assume that each of the subgroups has finite index. Therefore their intersection has finite index, too. It is clear that the intersection lies in the center of $G$ and therefore the group $G$ is central-by-finite. On the other hand, any group $G$ is covered by the abelian subgroups of the form $\langle g,Z(G)\rangle$, where $g\in G$. Therefore if the center $Z(G)$ has finite index in $G$, then $G$ is covered by finitely many abelian subgroups.
Thus, a group $G$ admits a finite covering by abelian subgroups if and only if $G$ is central-by-finite. This observation is due to Baer (cf \cite[4.16]{Robinson}).

Let $Z_i(G)$ denote the $i$th term of the upper central series of $G$. It is easy to check that the subgroup $\langle g,Z_i(G)\rangle$ is nilpotent of class at most $i$ for any $g\in G$. Therefore any group for which there exists $m$ such that $Z_m(G)$ has finite index is covered by finitely many nilpotent subgroups. In 1992 M.\,J.\ Tomkinson proved the converse: if a group $G$ admits a finite covering by nilpotent subgroups, then there exists some positive integer $m$ such that $Z_{m}(G)$ has finite index in $G$ \cite{Tom}. In view of Hall's theorem \cite{Hall} it follows that a group $G$ has a finite covering by nilpotent subgroups if and only if $G$ is finite-by-nilpotent. 
 
The above results apply to both abstract and profinite groups. Recently it was discovered that in the case of profinite groups similar results hold even with finite coverings replaced by countable ones. The following theorems were established in \cite{Sh2}.
 
 \begin{theorem}
 \label{nilp1}
 For a profinite group $G$ the following conditions are equivalent.
 \begin{itemize}
\item [(1)] The group $G$ is covered by countably many abelian subgroups;
 \item[(2)] The group $G$ has finite commutator subgroup;
\item [(3)] The group $G$ is central-by-finite.
 \end{itemize}
 \end{theorem}
 
\begin{theorem}
 \label{nilp2}
 For a profinite group $G$ the following conditions are equivalent.
 \begin{itemize}
\item [(1)] The group $G$ is covered by countably many nilpotent subgroups;
 \item[(2)] The group $G$ is finite-by-nilpotent;
\item [(3)] There exists a positive integer $m$ such that $Z_{m}(G)$ is open in $G$.
 \end{itemize}
 \end{theorem}
 
The proofs of Theorems \ref{nilp1} and \ref{nilp2} in \cite{Sh2} do not use the corresponding results on finite coverings of abstract groups. As a consequence of Theorems \ref{nilp1} and \ref{nilp2} we can observe that a profinite group $G$ admits a countable covering by abelian (respectively, nilpotent) subgroups if and only if $G$ admits a finite covering by subgroups with the respective property. 
 
\bigskip

\subsection*{On groups in which commutators are covered by countably many nilpotent subgroups}

Profinite groups that are covered by countably many nilpotent (or abelian) subgroups are understood reasonably well. Now it is time to have a look at groups in which commutators are covered by countably many nilpotent subgroups. The following theorem was obtained in \cite{Sh2}.
 
\begin{theorem}
\label{nilp3} Let $G$ be a profinite group. The following conditions are equivalent.
 \begin{itemize}
\item [(1)] The set of all commutators in $G$ is covered by countably many nilpotent subgroups;
\item [(2)]  The commutator subgroup $G'$ is finite-by-nilpotent;
\item [(3)]  There exists a positive integer $m$ such that $Z_{m}(G')$ is open in $G'$.
 \end{itemize} 
 \end{theorem}
Remark that no analogue of the above theorem for abstract groups is known. Thus, the next problem is very natural.
   
\begin{problem}\label{232} Let $G$ be an abstract group in which the  commutators are contained in a union of finitely many nilpotent subgroups. Is $G'$ finite-by-nilpotent? 
\end{problem}

Actually, the particular case of the above problem for covering by abelian subgroups was raised in \cite{FAS}:
 
\begin{problem}\label{233} Let $G$ be an abstract group in which the  commutators are contained in a union of finitely many abelian subgroups. Is the second commutator subgroup $G''$ finite? 
\end{problem}

Further, it is  natural to ask whether in case of abelian subgroups Theorem \ref{nilp1} can be made more specific.
\begin{problem} Let $G$ be a profinite group in which the commutators are covered by countably many abelian subgroups. Is the second commutator subgroup $G''$ necessarily finite?
\end{problem}
Finally, we formulate the related question for arbitrary multilinear commutator words.
\begin{problem} Let $w$ be a  multilinear commutator word and $G$ a profinite group in which the $w$-values are covered by countably many nilpotent subgroups. Is $w(G)$ finite-by-nilpotent? 
\end{problem}

\bigskip
 \subsection*{Coverings by soluble subgroups}

Suppose a group $G$ is covered by finitely many soluble subgroups. By Neumann's lemma $G$ must be virtually soluble. On the other hand, suppose that $G$ has a normal soluble subgroup $N$ of finite index. Then $G$ is covered by finitely many soluble subgroups of the form $\langle N,g\rangle$, where $g\in G$. Thus, $G$ is covered by finitely many soluble subgroups if and only if $G$ is virtually soluble.
 
The following result was obtained in \cite{AS2}.

\begin{theorem}
\label{cc3} Let  $w$ be a multilinear commutator word and $G$ a profinite group having countably many soluble subgroups whose union contains the set of $w$-values. Then $G$ is virtually soluble. 
\end{theorem}

A rather surprising fact that follows from Theorem \ref{cc3} is that in a profinite group the set of $w$-values is covered by countably many soluble subgroups if and only if the whole group $G$ is covered by finitely many soluble subgroups.

\section{Coverings of commutators by torsion subgroups, or subgroups of finite rank}
\subsection*{Coverings of commutators by torsion subgroups}

A group is periodic (torsion) if each of its cyclic subgroups has finite order. A group is called locally finite if each of its finitely generated subgroups is finite. Periodic profinite groups have received a good deal of attention in the past. In particular, using Wilson's reduction theorem \cite{Wil}, Zelmanov has been able to prove local finiteness of periodic compact groups \cite{Zel}. Earlier Herfort showed that there exist only finitely many primes dividing the orders of elements of a periodic profinite group \cite{her}. It follows that a periodic profinite group has a normal series of finite length each of whose factors is either a locally finite pro-$p$ group or a Cartesian product of isomorphic finite simple groups. It is a long-standing problem whether any periodic profinite group has finite exponent. Recall that a group $G$ has exponent $e$ if $x^e=1$ for all $x\in G$ and $e$ is the least positive integer with that property.

If a profinite group is covered by countably many periodic subgroups, by Baire's category theorem at least one of the subgroups is open and so the whole group is periodic. In the present section we will look at the results that deal with coverings of $w$-values by countably many periodic subgroups. The following theorem was proved in \cite{DMS}.

\begin{theorem}\label{cc1} Let $w$ be a multilinear commutator word and $G$ a profinite group  that has countably many periodic subgroups whose union contains all $w$-values in $G$. Then $w(G)$ is locally finite. 
\end{theorem}

Of course, under the hypotheses of Theorem \ref{cc1} all $w$-values in $G$ must have finite order. Therefore Theorem \ref{cc1} is related to the conjecture that if a profinite group $G$ admits a group-word $v$ such that all $v$-values have finite order, then the  verbal subgroup $v(G)$ is locally finite (see special cases stated in \cite{Sh3,khushu14} and similar problems 15.104 and 17.126 in the Kourovka Notebook \cite{kouro}). This conjecture is a natural generalization of the aforementioned Zelmanov's theorem on local finiteness of periodic profinite groups. Note that in this conjecture the orders of $v$-values are not assumed to be bounded.

Some particular cases of Theorem \ref{cc1} were established in other articles. The case of Theorem \ref{cc1} where $w=[x,y]$ was established in \cite{AS2}. The case where the set of $w$-values is covered by only finitely many periodic subgroups was established in \cite{AS1}. Since a solution to the problem whether any periodic profinite group has finite exponent is not in sight, it is natural to study in some detail the situation in which  the set of $w$-values is covered by countably many subgroups of finite exponent. A close inspection of the proof of Theorem \ref{cc1} in \cite{DMS} reveals that if $G$ is a profinite group that has countably many subgroups of finite exponent whose union contains all $w$-values, then $w(G)$ has finite exponent as well. However the proof in \cite{DMS} provides no clue to the natural question whether the exponent of $w(G)$ depends on the exponents of covering subgroups.

\begin{problem}\label{801} Let $w$ be a multilinear commutator word and $G$ a profinite group in which all $w$-values are contained in a union of countably many subgroups of finite exponent dividing $e$. Does $w(G)$ necessarily have an open subgroup of finite exponent dividing $e$?
\end{problem}

While finding appropriate techniques to handle the above problem remains elusive, we do have quantitative results on the bounds for the exponent of $w(G)$. The next theorem was proved in \cite{AS1}.

\begin{theorem}
\label{cc4} Let $e,k,s$ be positive integers and $G$ a profinite group that has  subgroups $G_1,G_2,\dots,G_s$ whose union contains all $\gamma_k$-values in $G$. Suppose that each of the subgroups $G_1,G_2,\dots,G_s$ has finite exponent dividing $e$. Then $\gamma_k(G)$ has finite $(e,k,s)$-bounded exponent.
\end{theorem}

The proof of Theorem \ref{cc4} given in \cite{AS1} is based on the Lie-theoretic techniques that Zelmanov created in his solution of the restricted Burnside problem \cite{ze1,ze2}. Theorem \ref{cc4} naturally suggests the problem whether a similar result holds for any multilinear commutator.

\begin{problem}\label{802} Let $e,k,s$ be positive integers and $w$ a multilinear commutator word of weight $k$. Suppose that $G$ is a profinite group having  subgroups $G_1,G_2,\dots,G_s$, each of finite exponent dividing $e$, whose union contains all $w$-values in $G$. Does $w(G)$ necessarily have finite $(e,k,s)$-bounded exponent?
\end{problem}

\bigskip
\subsection*{Coverings of commutators by subgroups of finite rank}
A profinite group is said to have finite rank $r$ if each of its subgroups can be generated by at most $r$ elements. In the case of abstract groups the definition is slightly different: an abstract group is said to have finite rank $r$ if each of its finitely generated subgroups can be generated by at most $r$ elements. It follows from results obtained independently by Guralnick \cite{gur} and Lucchini \cite{lucchi} that a profinite group has finite rank if and only if there exists an integer $r$ such that each of its Sylow subgroups is of rank at most $r$. Pro-$p$ groups of finite rank are now fairly well understood -- the theory of powerful pro-$p$ groups \cite{LM,LM2} provides the techniques and methods for handling such groups. In the present subsection we briefly survey results on coverings of commutators by subgroups of finite rank. The reader will see that currently all known results are in parallel with the results on coverings by periodic subgroups.

\begin{theorem}\label{cc15} Let $w$ be a multilinear commutator word and $G$ a profinite group  that has countably many subgroups of finite rank whose union contains all $w$-values in $G$. Then $w(G)$ has finite rank. 
\end{theorem}

\begin{theorem}
\label{cc45} Let $r,k,s$ be positive integers and $G$ a profinite group that has subgroups $G_1,G_2,\dots,G_s$ whose union contains all $\gamma_k$-values in $G$. Suppose that each of the subgroups $G_1,G_2,\dots,G_s$ has finite rank at most $r$. Then $\gamma_k(G)$ has finite $(r,k,s)$-bounded rank.
\end{theorem}
Theorems \ref{cc15} and \ref{cc45} were proved in \cite{DMS} and \cite{AS1}, respectfully. In the natural way they suggest open problems analogous to the problems mentioned in the previous subsection.

\begin{problem}\label{1801} Let $w$ be a multilinear commutator word and $G$ a profinite group in which all $w$-values are contained in a union of countably many subgroups of finite rank at most $r$. Does $w(G)$ necessarily have an open subgroup of finite rank at most $r$?
\end{problem}

\begin{problem}\label{1802} Let $r,k,s$ be positive integers and $w$ a multilinear commutator word of weight $k$. Suppose that $G$ is a profinite group having  subgroups $G_1,G_2,\dots,G_s$, each of finite rank at most $r$, whose union contains all $w$-values in $G$. Does $w(G)$ necessarily have finite $(r,k,s)$-bounded rank?
\end{problem}

\bigskip
\subsection*{Coverings of commutators by subgroups with finiteness conditions}

Since the results on coverings of commutators by subgroups of finite rank are similar to those on coverings of commutators by periodic subgroups, it is natural to look for a common approach to both. First of all we remark that both properties -- being periodic and being of finite rank -- are so called finiteness conditions, that is, conditions automatically satisfied by all finite groups. Further scrutiny reveals other properties common to all locally finite groups and groups of finite rank playing important role in the context of our subject. It seems likely that the above results are part of a more general phenomenon.

Indeed, let $\Sigma$ be a property of profinite groups such that:

1. Every finite group is a $\Sigma$-group;

2. The class of all $\Sigma$-groups is closed under taking subgroups, quotients and extensions;

3. If a profinite group $G$ is virtually soluble and $G/Z(G)$ is a $\Sigma$-group, then the commutator subgroup $G'$ is a $\Sigma$-group as well.

The following theorem was established in \cite{AS1}.
\begin{theorem}\label{gera} Let $w$ be a multilinear commutator word and $G$ a profinite group  that has finitely many $\Sigma$-subgroups whose union contains all $w$-values in $G$. Then $w(G)$ has the property $\Sigma$. 
\end{theorem}

In the case where the coverings are finite both Theorem \ref{cc1} and Theorem \ref{cc15} can be deduced from Theorem \ref{gera}. We do not know if Theorem \ref{gera} remains correct if the hypothesis that the covering is finite is replaced by a (seemingly weaker) hypothesis that $G$ has countably many $\Sigma$-subgroups covering all $w$-values. We therefore state this as a conjecture.

\begin{conjecture}\label{gerat} Let $\Sigma$ be as above and $w$ be a multilinear commutator word and $G$ a profinite group  that has countably many $\Sigma$-subgroups whose union contains all $w$-values in $G$. Then $w(G)$ necessarily has the property $\Sigma$. 
\end{conjecture}
    
\section{On conciseness of words in profinite groups}

A group-word $w$ is said to be concise if whenever the set of its values in a group $G$ is finite it always follows that the verbal subgroup $w(G)$ is finite. More generally, a word $w$ is said to be concise in a class of groups $X$ if whenever the set of its values in a group $G \in X$ is finite it always follows that $w(G)$ is finite. In 1957 P. Hall asked whether every word is concise and for a period of time this was an open problem attracting interest of many group-theorists. In 1992 Ivanov proved that this problem has a negative solution \cite{Ivan} (see also \cite[p.\ 439]{Ol}). On the other hand, many relevant words are known to be concise. 

For instance, it is an easy observation by P. Hall that every non-com\-mu\-ta\-tor word is concise (see e.g. \cite[Lemma 4.27]{Robinson}). It was shown in \cite{wilson} that the multilinear commutator words are concise (see also \cite{FAM}). Merzlyakov showed that every word is concise in the class of linear groups \cite{Merz} while Turner-Smith proved that every word is concise in the class of residually finite groups all of whose quotients are again residually finite \cite{TS}. It was shown in \cite{AS5} that if $w$ is a multilinear commutator word and $n$ is a prime-power, then the word $w^n$ is concise in the class of residually finite groups. In \cite{GS} the words $[\dots[x_1^{n_1},x_2]^{n_2},\dots,x_k]^{n_k}$ were shown to be concise in residually finite groups.

There is an open problem, due to Jaikin-Zapirain \cite{Jaik}, whether every word is concise in the class of profinite groups. Of course, every word that is concise in residually finite groups is also concise in profinite ones. An interesting phenomenon was discovered in \cite{DMS2}. It suggests that perhaps the very definition of conciseness in profinite groups should be relaxed. 

\begin{conjecture}\label{A} Assume that the word $w$ has only countably many values in a profinite group $G$. Then the verbal subgroup $w(G)$ is finite.
\end{conjecture}

It follows easily from \cite[Theorem 1.3]{Jaik} that the above conjecture holds true whenever $G$ is a compact $p$-adic analytic group. The article \cite{DMS2} provides further evidence in support of the conjecture. In particular, the results obtained in \cite{DMS2} confirm the conjecture in the cases where $w$ is a multilinear commutator, or $w=x^2$, or $w=[x^2,y]$. 

The proof of Conjecture \ref{A} for the word $w=\delta_k$ is rather short (modulo some results described in the previous section) and so we give it here. We require the following simple lemma.

\begin{lemma}\label{boundedly-many}  Let $m\ge1$ and $k\ge0$ be integers. There exists a number $t=t(m,k)$ depending on $m$ and $k$ only, such that if $G$ is an $m$-generated profinite group, then every $\delta_{k}$-commutator in elements of $G'$ is a product of at most $t$ elements which are $\delta_{k+1}$-commutators in elements of $G$.
\end{lemma}
\begin{proof}
We use induction on $k$. Let $H$ be the commutator subgroup of $G$. As $G$ is finitely generated,  
a theorem of Nikolov and Segal \cite{nisega} tells us that every element of $H$ is a product
of $r$ commutators where $r$ depends only on $m$. So for $k=0$ the result follows. 

Assume that $k\geq1$ and let $u$ be a $\delta_k$-commutator in elements of $H$. Write $u=[h_1,h_2]$ where $h_1,h_2$ are $\delta_{k-1}$-commutators in  elements of $H$. By induction, $h_1,h_2$ are both  product of at most $t_1=t(m,k-1)$  $\delta_{k}$-commutators in elements of $G$. Using the well-known commutator identities $[xy,z]=[x,z]^y\,[y,z]$, $[x,yz]=[x,z]\,[x,y]^z$ we can decompose $u$ as a product of at most $t_1^2$ commutators $[a,b]$ where $a$ and $b$ are $\delta_{k}$-commutators in elements of $G$. The lemma follows.
\end{proof}

\begin{theorem}\label{delta} Let $k$ be a nonnegative integer and $G$ a profinite group with only countably many $\delta_k$-commutators. Then the $k$th commutator subgroup $G^{(k)}$ is finite.
\end{theorem}
\begin{proof} We use induction on $k$. Of course the result holds when $k=0$ because infinite profinite groups are uncountable. 

By assumption, all $\delta_k$-values are covered by countably many procyclic subgroups. Thus, by Theorem \ref{cc15}, $G^{(k)}$ has finite rank. If every $\delta_k$-value has finite order, then by  Theorem \ref{cc1}, $G^{(k)}$ is locally finite, and hence finite. So we now assume that there exist elements $a_1,\ldots,a_{2^k}$ such that $\delta_k(a_1,\ldots,a_{2^k})$ has infinite order. Without loss of generality we can assume that $G=\langle a_1,\ldots,a_{2^k} \rangle$. 

As $G$ is finitely generated, by Lemma \ref{boundedly-many}, every $\delta_{k-1}$-value of $H=G'$ is a product of boundedly many $\delta_{k}$-values in elements of $G$. Therefore $H$ has only countably many $\delta_{k-1}$-values and by induction $H^{(k-1)}$ is finite. In particular $\delta_k(a_1,\ldots,a_{2^k})$ has finite order, a contradiction. 
\end{proof}

Conjecture \ref{A} seems rather hard to deal with. Even the case of the word $w=x^p$, where $p$ is an odd prime, is unclear. Perhaps the next conjecture should be more tractable.

\begin{conjecture}\label{B} Assume that the word $w$ has only countably many values in a profinite group $G$. Then $G$ satisfies a nontrivial group identity.
\end{conjecture}

As a small step that eventually could be helpful in handling Conjectures \ref{A} and \ref{B} we will now show that under the above hypotheses the group $G$ satisfies a coset identity. Recall that a nontrivial word $w=w(x_1,\dots,x_n)$ is said to be a coset identity of a profinite group $G$ if there exists an open subgroup $H\leq G$ and cosets $g_1H,\dots,g_nH$ such that $w(g_1H,\dots,g_nH)=\{1\}$.

Thus, suppose that the word $w=w(x_1,\dots,x_n)$ has only countably many values $u_1,u_2,\dots$ in a profinite group $G$. For each $i$ set $$S_i=\{(g_1,\dots,g_n)\in G\times\dots\times G;\ w(g_1,\dots,g_n)=u_i\}.$$ The sets $S_i$ are closed in $G\times\dots\times G$ and cover $G\times\dots\times G$. Therefore, by Baire category theorem at least one of the sets $S_i$ has non-empty  interior. It follows that for some $k$ the group $G$ contains an open subgroup $H$ and elements $a_1,\dots,a_n$ such that $w(a_1H,\dots,a_nH)=\{u_k\}$. Set $$v(x_1,\dots,x_{2n})=w(x_1,\dots,x_n)w(x_{n+1},\dots,x_{2n})^{-1}.$$ It follows that the word $v(x_1,\dots,x_{2n})$ is a coset identity in $G$.

We conclude this article by showing that Conjecture \ref{B} is true when $w=x^p$, where $p$ is a prime. Indeed, we argue as above and conclude that for some $k$ the group $G$ contains an open normal subgroup $H$ and an element $a$ such that $(ah)^p=u_k$ for all $h\in H$. Let $N=\langle {u_k}^G\rangle$ be the minimal normal subgroup of $G$ containing $u_k$. By \cite[Lemma 3.1]{DMS2}, $C_G(u_k)$ is open in $G$ and therefore $N$ is central-by-finite. Hence, it is sufficient to show that $G/N$ satisfies a nontrivial identity and so we assume that $u_k=1$. 
 If $a\in H$, then $H$ is of exponent $p$ and the result is immediate. Therefore we assume that $a\not\in H$.
It follows that $H$ admits a split automorphism of prime order $p$. Kegel proved that finite groups admitting such an automorphism are nilpotent \cite{Kegel}. It follows that $H$ is pronilpotent. If $x,y\in H$, we can find an $a$-invariant subgroup $K\leq H$ containing both $x$ and $y$ and generated by at most 2$p$ elements. Khukhro's theorem \cite{khukhro} (see also \cite[Theorem 7.2.1]{bookhukhro}) tells us that $K$ is nilpotent with class bounded by a function of $p$ only. Of course, this is sufficient to conclude that $H$ satisfies a nontrivial identity and so does $G$.

\end{document}